\documentclass[reqno,11pt]{amsart} 

\usepackage{amsmath,amssymb,amsthm,mathtools}
\usepackage{mathrsfs}       % Script fonts
\usepackage{hyperref}
\usepackage[nameinlink]{cleveref}
\usepackage{xcolor}         % Colored text
\usepackage{tikz}           % Drawing
\usepackage{verbatim}       % Commenting
\usepackage{setspace}       % Line spacing
\usepackage[normalem]{ulem}
\title[On a Weiss-type Almost Monotonicity Formula]{On a Weiss-type Almost Monotonicity Formula}

\author[A. Sobral]{Aelson Sobral}
\address{Applied Mathematics and Computational Sciences (AMCS), Computer, Electrical and Mathematical Sciences and Engineering Division (CEMSE), King Abdullah University of Science and Technology (KAUST), Thuwal, 23955-6900, Kingdom of Saudi Arabia}
\email{aelson.sobral@kaust.edu.sa}

\newtheorem{theorem}{Theorem}[section]
\newtheorem{lemma}{Lemma}[section]
\newtheorem{corollary}{Corollary}[section]
\newtheorem{proposition}{Proposition}[section]

\newtheorem{remark}{Remark}[section]

\newcommand{\intav}[1]{\mathchoice 
  {\mathop{\vrule width 6pt height 3 pt depth -2.5pt \kern -8pt \intop}\nolimits_{\kern -6pt#1}} 
  {\mathop{\vrule width 5pt height 3 pt depth -2.6pt \kern -6pt \intop}\nolimits_{#1}}
  {\mathop{\vrule width 5pt height 3 pt depth -2.6pt \kern -6pt \intop}\nolimits_{#1}}
  {\mathop{\vrule width 5pt height 3 pt depth -2.6pt \kern -6pt \intop}\nolimits_{#1}}}

\numberwithin{equation}{section}
\begin{document}

\subjclass[2020]{Primary 35R35. Secondary 35A21}

% 35R35 Free boundary problems for PDEs 

% 35A21 Singularity in context of PDEs

\keywords{Monotonicity formula, Free boundary problems, Non-differentiable functionals}

\begin{abstract} 
We establish a Weiss-type almost-monotonicity formula for a broad class of variable-coefficient energy functionals, assuming only minimal regularity of the coefficients. As an application, we classify blow-up limits for the Alt--Phillips problem with variable coefficients under significantly weaker regularity hypotheses than those imposed in Ara\'ujo et al. [Calc. Var. Partial Differential Equations, 65, no.~1, Paper No.~24 (2026)]. Moreover, by means of a distinct argument, we extend the corresponding free-boundary regularity result. We conclude with a discussion of further extensions, including two-phase analogues.
\end{abstract}  

\date{\today}

\maketitle

%\tableofcontents

\section{Introduction} \label{sct intro}

For a bounded domain $\Omega \subset \mathbb{R}^d$ and a measurable, locally bounded function $\sigma \colon \Omega \times \mathbb{R}^+ \to \mathbb{R}^+$, we consider almost minimizers to the functional
\begin{equation}\label{main functional}
    \mathcal{J}(u,\Omega) \coloneqq \int_{\Omega} \left( \frac{1}{2} \, \langle A(x)\nabla u, \nabla u \rangle + \sigma(x,u)\right) \, dx,
\end{equation}
where $A(x)$ is a uniformly elliptic matrix-valued function. In this work, we establish a new Weiss-type almost monotonicity formula, which proves to be an effective tool for classifying homogeneity in a broad class of problems. Several applications illustrating the scope of this result are also discussed.

Monotonicity formulae of Weiss-type, when available, are key to analyzing the fine structure of the free-boundary. One of the major developments is the homogeneity approach introduced by Weiss~\cite{W2}, which constitutes the main object of study in this paper. Further extensions and generalizations include the Alt--Caffarelli--Friedman monotonicity formula~\cite{ACF}, the Monneau-type monotonicity formula \cite{MP}, and the Caffarelli--Jerison--Kenig result \cite{CJK}.

The goal of this paper is to establish a Weiss-type monotonicity formula for a broad class of energy functionals of the form \eqref{main functional}. We will specify the precise assumptions on the ingredients later, but the key point is that both terms in the energy involve variable coefficients, which may complicate the analysis of the associated free boundary. Since these coefficients possess only limited regularity, minimizers are not expected to satisfy the standard Weiss monotonicity formula. Instead, we derive an almost (or quasi) monotonicity formula, in which an error term quantifies the deviation from monotonicity depending on how far the coefficients are from being constant. Under minimal regularity assumptions, this formulation still allows for the classification of homogeneous blow-ups.

The main idea is conceptually simple and relies on the optimal regularity of the solution near free boundary points. The strategy consists in \emph{freezing the coefficients}: for $x_0 \in F(u)$, we consider
\[
    \mathcal{J}_{x_0}(u,\Omega) \coloneqq \int_\Omega \left( \frac{1}{2} \, \langle A(x_0)\nabla u, \nabla u \rangle + \sigma(x_0,u) \right)\,dx,
\]
and use optimal regularity to compare this with the original functional, by testing against a suitable competitor and then comparing back to the ``frozen-at-$x_0$'' energy functional. This back-and-forth comparison requires some regularity of the coefficients; however, once optimal regularity for minimizers is established, the assumptions on the coefficients can be significantly weakened. This is precisely why the resulting formula is of \emph{almost} monotonicity type. We focus primarily on one-phase Alt--Phillips-type problems, though the same approach can be extended to other kinds of variable coefficient functionals with possibly two-phases, see Section \ref{sct:some-extensions}.

The terminology \emph{almost monotonicity} (or \emph{quasi monotonicity}) appears in the literature in various contexts, though typically for problems of a different nature. We refer to \cite{EP} and \cite{MP}, where almost monotonicity formulas were established for both elliptic and parabolic equations, extending the results of \cite{CJK} and the earlier works of Caffarelli and Kenig. In the context of Weiss- and Monneau-type formulas, we mention \cite{FGS}, which treats the obstacle problem with Lipschitz coefficients and employs these ideas to extend the classical analyses of Caffarelli, Monneau, and Weiss. To the best of our knowledge, this is the first work where such an approach is developed for Alt--Phillips-type problems under very low regularity assumptions on the coefficients.

Even though our class of energy functionals is fairly general, and we work with almost minimizers, our main focus lies on the specific case
\[
	\int_{\Omega} \left( \frac{|\nabla u|^2}{2} + \delta(x)\,u^{\gamma(x)} \right)\,dx,
\]
where $\delta$ is bounded away from zero and $\gamma(x) \in [\min \gamma, \max \gamma] \Subset (0,1)$. This problem was previously studied in \cite{ASTU1}, while the version with variable coefficients was investigated in \cite{ASTU2}. In those works, the full regularity program was carried out, from the analysis of the solutions up to the regularity of the free boundary. In the present paper, we address the final step of that program, substantially relaxing the regularity assumptions on both $\delta$ and $\gamma$ to a more natural setting:
\[
	\delta \in C^{0,\text{Dini}} \quad \text{and} \quad \gamma \in C^{0,\log\text{-Dini}}.
\]
In contrast, \cite{ASTU1} required $\delta,\gamma \in W^{1,q}$ for some $q>d$, a considerably stronger assumption. Determining whether our monotonicity formula applies in full generality lies beyond the scope of this paper and is left for future investigation. In particular, we establish the homogeneity of blow-ups for the problem considered in \cite{ASTU2}, as well as for the Alt--Caffarelli functional with variable coefficients and its two-phase versions.

The approach developed in this paper suggests several natural directions for future work. A first question is whether the strategy of \cite{AKS} can be adapted to the present setting and, if so, what regularity assumptions it would impose on the coefficients. This is not immediate, since in our problem the right-hand side also contains variable coefficients. Another direction is to identify the weakest coefficient regularity under which free-boundary regularity persists, both for the problem studied in \cite{ASTU2} and for the Alt--Caffarelli problem discussed in Section~\ref{sct:some-extensions}. A further challenging problem is to develop a monotonicity-based framework yielding a \(p\)-Laplacian analogue of the monotonicity formula adapted to the problem studied in \cite{STU}.

The paper is organized as follows. We give the formula and proof in Section \ref{sct:almost-monotone-formula}. In Section \ref{sct:Alt-philips-varying}, we give two applications of this new almost monotonicity formula. Finally, in Section \ref{sct:some-extensions}, we show how the formula can be extended to other scenarios.

\section{Definitions and preliminaries}\label{sct:prelim}

This section is devoted to establishing the mathematical setting of the paper. We fix notation and collect preliminary definitions and known results that will be used in the sequel.

\subsection{Mathematical setup and assumptions}

Given a domain $\Omega \subset \mathbb{R}^d$, we say that $u \geq 0$ is an \emph{almost minimizer} to
\begin{equation}\label{def:main functional}
    \mathcal{J}(w) \coloneqq \int \left( \frac{1}{2} \, \langle A(x)\nabla w, \nabla w \rangle + \delta(x)\,w^{\gamma(x)} \right)\,dx,
\end{equation}
in $\Omega$, if for every $r>0$ and $x_0 \in \mathbb{R}^d$ such that $B_r(x_0) \subset \Omega$, we have
\[
    \mathcal{J}(u,B_r(x_0)) \le (1+\varrho(r))\,\mathcal{J}(v,B_r(x_0))
\qquad \forall\, v \in u + H^1_0(B_r(x_0)).
\]
where
\[
    \mathcal{J}(w,\mathcal{O}) \coloneqq \int_{\mathcal{O}} \left( \frac{1}{2} \, \langle A(x)\nabla w, \nabla w \rangle + \delta(x)\,w^{\gamma(x)} \right)\,dx.
\]
When $\varrho \equiv 0$, we say that $u$ is a \emph{local minimizer}.

Central to the analysis in this paper is the notion of the \emph{Alt--Phillips functional frozen at $x_0$}. For a given $x_0 \in \mathbb{R}^n$, we set
\[
	\mathcal{J}_{x_0}(w,\mathcal{O}) \coloneqq \int_\mathcal{O} \left( \frac{1}{2}\, \langle A(x_0)\nabla w, \nabla w \rangle + \delta(x_0)\,w^{\gamma(x_0)} \right)\,dx.
\]

\medskip

The basic assumptions on $A$, $\delta$, $\gamma$, and $\varrho$ are the following. The mapping
$A\colon \Omega \mapsto \mathbb{S}_{\mathrm{sym}}(d)$ is continuous, where $\mathbb{S}_{\mathrm{sym}}(d)$ denotes the space of symmetric $d \times d$ matrices. We assume that $A$ is uniformly elliptic: there exists $\Lambda \in (0,1]$ such that
\begin{equation*}
     \langle A(x)\xi, \xi \rangle \in \big[\Lambda,\Lambda^{-1}\big] \quad \text{for every} \quad x \in \Omega \quad \text{and} \quad |\xi| = 1.
\end{equation*}
Moreover,
\[
    \big[\inf_{\Omega}\delta,\ \sup_{\Omega}\delta\big]\Subset(0,\infty),
    \qquad
    \big[\inf_{\Omega}\gamma,\ \sup_{\Omega}\gamma\big]\Subset(0,1),
\]
so that $\delta$ is bounded away from $0$, and $\gamma$ stays a positive distance away from both $0$ and $1$. Finally, we assume that the gauge function $\varrho$ is a modulus of continuity.
 
We denote by $\omega_A$, $\omega_\delta$, and $\omega_\gamma$ the moduli of continuity of $A$, $\delta$, and $\gamma$, respectively. The precise continuity assumptions imposed on $A$, $\delta$, $\gamma$, and $\varrho$ will be specified in the statements of the main results. Throughout the paper, the notation $\lesssim$ means inequality up to a multiplicative universal constant; here and in what follows, \emph{universal} indicates dependence only on the dimension $d$, the ellipticity constant $\Lambda$, and the bounds for $\delta$ and $\gamma$.

We also assume a quantitative growth control for $u$ near free boundary points. More precisely, for every $x_0\in \partial\{u>0\}\cap\Omega$ and every ball $B_r(x_0)\subset\Omega$, we require
\begin{equation}\label{assumptions: growth for min}
    \sup_{x\in B_r(x_0)}\Big(u(x)+|\nabla u(x)|^{\frac{2}{\gamma(x_0)}}\Big) \lesssim r^{\beta(x_0)},
\end{equation}
where
\[
    \beta(x_0) \coloneqq \frac{2}{2-\gamma(x_0)}.
\]
Such growth estimates are available in several relevant settings; we refer the reader to \cite{ASTU1,ASTU2} for further discussion.

\subsection{Free boundary and pointwise homogeneity}

Throughout the paper, we focus on one-phase almost minimizers $u$ to the functional \eqref{def:main functional} in $\Omega = B_2$. The Weiss-type almost monotonicity formula proved in the next section should also extend to the two-phase setting, provided the corresponding growth assumptions are available, as discussed in the introduction.

The growth assumption \eqref{assumptions: growth for min} is crucial for the fine analysis of the free boundary
\[
    F(u)\coloneqq \partial\{u>0\}\cap B_2.
\]
In particular, it allows us to introduce blow-up sequences at free boundary points $x_0\in F(u)$ by setting
\[
    u_r(x)\coloneqq r^{-\beta(x_0)}\,u(x_0+rx),\qquad r\in(0,1).
\]
By the scaling properties of the functional, together with nondegeneracy and the compactness of (almost) minimizers, for every sequence $r_k\downarrow 0$ there exists a subsequence (still denoted $r_k$) such that $u_{r_k}\to u_0$ locally in $\mathbb{R}^n$, for some nontrivial limit $u_0\not\equiv 0$.

The Weiss-type almost monotonicity formula proved in the next section implies that every blow-up limit $u_0$ is homogeneous about $0$, with homogeneity exponent $\beta(x_0)$. More precisely,
\[
    u_0(\lambda x)=\lambda^{\beta(x_0)}\,u_0(x)
    \qquad\text{for all }\lambda\ge 0.
\]
After the blow-up procedure, we may distinguish the free boundary points into two kinds. We say that $x_0$ is a \emph{regular point} of $F(u)$ if at least one blow-up limit is of the form
\[
    u_0(x)=\vartheta(x_0)\,(x\cdot \nu)_+^{\beta(x_0)},\qquad \nu\in \partial B_1,
\]
and we refer to such profiles as \emph{trivial cones}. The constant $\vartheta(x_0)$ is chosen so that $u_0$ solves the Alt--Phillips equation with coefficients frozen at $x_0$, namely
\[
    \Delta u_0=\delta(x_0)\gamma(x_0)\,u_0^{\gamma(x_0)-1} \quad \text{in} \quad \{u_0>0\}.
\]
The precise expression for $\vartheta(x_0)$ is given in Section \ref{sct:Alt-philips-varying}.

We denote the set of regular points by $\mathcal{R}(u)$ and define the \emph{singular set} by
\[
    \Sigma(u)\coloneqq F(u)\setminus \mathcal{R}(u),
\]
so that
\[
    F(u)=\mathcal{R}(u)\cup \Sigma(u).
\]
Finally, when a $\beta(x_0)$-homogeneous blow-up is not a trivial cone, we call it a \emph{nontrivial minimal cone}. We emphasize that this notion depends on the base point $x_0$ at which the blow-up is taken.

\subsection{Known results}

We state a Lemma that gives compactness of minimizers, as well as convergence of free boundaries. Its proof is rather standard, but we bring it here for completeness.

\begin{lemma}\label{lemma:compactness}
For each $k\in\mathbb N$, let $u_k$ be an almost minimizer in $\Omega$ to the functional
\[
    \mathcal J_k(v) \coloneqq \int \left(\frac{|\nabla v|^2}{2} + \delta_k(x)\,v^{\gamma_k(x)}\right)\,dx.
\]
Assume
\[
u_k \to u \quad \text{in } L^\infty_{\mathrm{loc}}(\Omega)\cap H^1_{\mathrm{loc}}(\Omega),
\quad
\delta_k \to \delta \ \text{in } L^\infty_{\mathrm{loc}}(\Omega),
\quad
\gamma_k \to \gamma \ \text{in } L^\infty_{\mathrm{loc}}(\Omega).
\]
Then, $u$ is an almost minimizer in $\Omega$ to the functional
\[
    \mathcal J(v) \coloneqq \int \left(\frac{|\nabla v|^2}{2} + \delta(x)\,v^{\gamma(x)}\right)\,dx.
\]
Moreover, if $u_k$ is uniformly nondegenerate, meaning that there is a modulus of continuity $\varsigma$ such that
\[
    \sup_{B_r(x)} u_k \geq c\,\varsigma(r), \quad \text{for} \quad x \in \overline{\{u>0\}}\cap \Omega, \quad \text{and} \quad B_r(x) \Subset \Omega,
\]
then
\[
    \chi_{\{u_k>0\}} \to \chi_{\{u>0\}} \quad\text{in }L^1_{loc}(\Omega),
\]
and the free boundaries $F(u_k)$ converge locally to $F(u)$ in the Hausdorff distance.
\end{lemma}

\begin{proof}
Fix a ball $B_r(x_0) \subset \Omega$ and let $v\in u+H^1(B_r(x_0))$ be arbitrary. For each $k$, let $\varphi_k$ denote the harmonic replacement of $u_k-u$ in $B_r(x_0)$, namely the unique solution of
\[
    \Delta\varphi_k=0 \quad\text{in }B_r(x_0), \qquad \varphi_k=u_k-u \quad\text{on }\partial B_r(x_0).
\]
Define the competitor
\[
    v_k \coloneqq v+\varphi_k.
\]
Since $v \in u + H^1(B_r(x_0))$, we obtain $v_k\in u_k+H^1(B_r(x_0))$. Therefore, using $v_k$ as a competitor in the almost minimality of $u_k$ yields
\begin{equation}\label{ineq:almost min} 
\begin{aligned} 
    &\int_{B_r(x_0)} \left(\frac{|\nabla u_k|^2}{2} + \delta_k(x)\,u_k^{\gamma_k(x)}\right)\,dx \\ 
    &\le (1 + \varrho(r)) \int_{B_r(x_0)} \left(\frac{|\nabla v_k|^2}{2} + \delta_k(x)\,v_k^{\gamma_k(x)}\right)\,dx . 
\end{aligned} 
\end{equation}

We claim that
\[
    v_k \to v \quad\text{in } L^\infty(\overline{B}_r(x_0))\cap H^1(B_r(x_0)).
\]
Indeed, by the maximum principle for harmonic functions,
\[
    \|\varphi_k\|_{L^\infty(\overline{B}_r(x_0))} \le \|u_k-u\|_{L^\infty(\partial B_r(x_0))} \le \|u_k-u\|_{L^\infty(\overline{B}_r(x_0))}\to 0,
\]
hence $\varphi_k\to 0$ uniformly in $\overline{B}_r(x_0)$, and thus $v_k\to v$ in $L^\infty(\overline{B}_r(x_0))$. Moreover, since $\varphi_k$ minimizes the Dirichlet energy among all functions with the same boundary trace $u_k-u$ on $\partial B_r(x_0)$, we have
\[
    \int_{B_r(x_0)}|D\varphi_k|^2\,dx \le \int_{B_r(x_0)}|\nabla(u_k-u)|^2\,dx \to 0,
\]
which implies $\varphi_k\to 0$ in $H^1(B_r(x_0))$, and so $v_k\to v$ in $H^1(B_r(x_0))$.

Finally, using the convergences $u_k\to u$, $v_k\to v$, $\delta_k\to\delta$, and $\gamma_k\to\gamma$, we may pass to the limit in \eqref{ineq:almost min} to obtain
\[
    \begin{aligned} 
    &\int_{B_r(x_0)} \left(\frac{|\nabla u|^2}{2} + \delta(x)\,u^{\gamma(x)}\right)\,dx \\ 
    &\le (1 + \varrho(r)) \int_{B_r(x_0)} \left(\frac{|\nabla v|^2}{2} + \delta(x)\,v^{\gamma(x)}\right)\,dx . 
    \end{aligned} 
\]
Since $B_r(x_0)\subset\Omega$ and $v\in u+H^1(B_r(x_0))$ were arbitrary, this shows that $u$ is an almost minimizer (with the same gauge $\varrho$) of $\mathcal J$ in $\Omega$.

Next, we show that
\[
    \chi_{\{u_k>0\}} \to \chi_{\{u>0\}} \qquad \text{in } L^1_{\mathrm{loc}}(\Omega).
\]
Since $0\le \chi_{\{u_k>0\}}\le 1$, the dominated convergence theorem reduces the claim to proving
\[
    \chi_{\{u_k>0\}}(x)\to \chi_{\{u>0\}}(x)\qquad \text{for a.e. }x\in\Omega.
\]
By the uniform nondegeneracy of $u_k$ (and hence of the limit $u$), the free boundary $\partial\{u>0\}$ has Lebesgue measure zero, so we may restrict the analysis to
\[
    x \in \Omega\setminus \partial\{u>0\}.
\]
If $u(x)>0$, then, by continuity, there exists $r_x>0$ such that $B_{r_x}(x)\subset\{u>0\}\cap\Omega$. Uniform convergence $u_k\to u$ implies that $u_k>0$ in $B_{r_x/2}(x)$ for all sufficiently large $k$, hence $\chi_{\{u_k>0\}}(x)=1=\chi_{\{u>0\}}(x)$.

If instead $x\in\{u=0\}^\circ$, there exists $r_x>0$ such that $B_{r_x}(x)\subset\{u=0\}\cap\Omega$. We claim that $B_{r_x/2}(x)\subset\{u_k=0\}$ for all large $k$. Otherwise, along a subsequence, there would exist $y_k\in B_{r_x/2}(x)$ with $u_k(y_k)>0$, and so $B_{r_x/2}(x)$ intersects $F(u_k)$. If we take
\[
    x_k\in B_{r_x/2}(x)\cap\partial\{u_k>0\},
\]
Then, by uniform nondegeneracy at the free boundary point $x_k$, we have
\[
    \sup_{B_{r_x}(x)} u_k \ge \sup_{B_{r_x/2}(x_k)} u_k \ge c\,\varsigma(r_x).
\]
However, since $u\equiv 0$ in $B_{r_x}(x)$ and $u_k\to u$ uniformly, we have $\sup_{B_{r_x}(x)}u_k\to 0$, a contradiction. Hence $B_{r_x/2}(x)\subset\{u_k=0\}$ for all large $k$, and therefore $\chi_{\{u_k>0\}}(x)=0=\chi_{\{u>0\}}(x)$. In conclusion, $\chi_{\{u_k>0\}}(x)\to \chi_{\{u>0\}}(x)$ for a.e. $x\in\Omega$.

Finally, we prove the convergence of the free boundaries. Let $K\Subset\Omega$. From the previous argument, if $B_r(x)\subset\{u>0\}$ (respectively $B_r(x)\subset\{u=0\}$), then for $k$ large enough $B_{r/2}(x)\subset\{u_k>0\}$ (respectively $B_{r/2}(x)\subset\{u_k=0\}$). Hence, for every $\varepsilon>0$ and all large $k$,
\[
    F(u_k)\cap K \subset \{y\in K \colon \mathrm{dist}(y,\partial\{u>0\})<\varepsilon\}.
\]
Conversely, if $x\in F(u)\cap K$ and along a subsequence $\mathrm{dist}(x,F(u_k))\geq\varepsilon$, then either $B_\varepsilon(x)\subset\{u_k>0\}$ or $B_\varepsilon(x)\subset\{u_k=0\}$ along that subsequence. By uniform convergence this would force $B_{\varepsilon/2}(x)\subset\{u>0\}$ or $B_{\varepsilon/2}(x)\subset\{u=0\}$, contradicting $x \in F(u)$. Thus, for each $x \in F(u) \cap K$ there exist $x_k\in F(u_k)$ with $x_k\to x$. These two inclusions imply the local Hausdorff convergence of $F(u_k)$ to $F(u)$.
\end{proof}

\section{Weiss almost monotonicity}\label{sct:almost-monotone-formula}

In this section, we state and prove the main result of the paper: a Weiss-type almost monotonicity formula that applies to almost minimizers to \eqref{def:main functional}. As an immediate consequence, we obtain the homogeneity of blow-up limits. The approach developed here significantly reduces the regularity assumptions required on the coefficients $\delta$ and $\gamma$ compared to \cite{ASTU1}. We focus here on the one-phase case and explain in Section \ref{sct:some-extensions} how the argument extends to the two-phase setting.

\begin{theorem}\label{thm:almost-monotonicity}
Let $u \geq 0$ be an almost minimizer to the functional \eqref{def:main functional} in $B_2$. Assume further that
\[
	\omega_A(\cdot) + \omega_\delta(\cdot) + \varrho(\cdot) + |\ln(\cdot)|\omega_\gamma(\cdot) \in C^{0,dini} ([0,2])\quad \text{and} \quad A(x_0) = I_d,
\]
for some $x_0 \in F(u)\cap B_1$. Then, the function $\mathcal{A}_{u,x_0}(r)$ defined by
\[
     r^{-(d + 2(\beta(x_0) - 1))}\mathcal{J}_{x_0}(u,B_r(x_0)) - \frac{1}{2}\beta(x_0)r^{-(d + 2\beta(x_0) -1)}\int_{\partial B_r(x_0)}u^2\, d\mathcal{H}^{d-1},
\]
is almost monotone, meaning that there exists a nonnegative real function $g \in L^1(0,2)$ such that
\[
	\frac{d}{dr} \mathcal{A}_{u,x_0}(r) \geq -g(r) \quad \text{for} \quad r \in (0,1/2).
\]
\end{theorem}

\begin{proof}
Without loss of generality, assume $x_0=0$ and set
\[
    \beta_0 \coloneqq \beta(0), \qquad \gamma_0 \coloneqq \gamma(0).
\]
A direct computation yields the following expression for $\frac{d}{dr}\mathcal{A}_{u,0}(r)$
\begin{align*}
& -(d+2(\beta_0-1))\,r^{-(d+2\beta_0-1)}\mathcal{J}_{0}(u,B_r)
   + r^{-(d+2(\beta_0-1))}\mathcal{J}_{0}(u,\partial B_r) \\
&\quad - \beta_0\,r^{-(d+2\beta_0)}
\int_{\partial B_r}\bigl( r\,u\,\partial_\nu u - \beta_0 u^2 \bigr)\,d\mathcal{H}^{d-1},
\end{align*}
where, with a slight abuse of notation,
\[
    \mathcal{J}_{0}(u,\partial B_r) \coloneqq \int_{\partial B_r}\left(\frac{|\nabla u|^2}{2} + \delta(0)\,u^{\gamma_0}\right)\,d\mathcal{H}^{d-1}.
\]
On $\partial B_r$, we decompose the gradient into normal and tangential parts,
\[
    \nabla u = (\partial_\nu u)\,\nu + \nabla_\tau u,
\]
and we introduce
\begin{align*}
\mathcal{B}(u) \coloneqq\;
& r^{-(d+2(\beta_0-1))}\int_{\partial B_r}\left(\frac{|\nabla_\tau u|^2}{2}+\delta(0)\,u^{\gamma_0}\right)\,d\mathcal{H}^{d-1} \\
&\;+\frac{1}{2}\beta_0^2\,r^{-(d+2\beta_0)}\int_{\partial B_r}u^2\,d\mathcal{H}^{d-1}.
\end{align*}
With this notation, we can rewrite
\begin{equation}\label{eq:derivative-formula}
    \frac{d}{dr}\mathcal{A}_{u,0}(r) = \mathcal{B}(u) - (d+2(\beta_0-1))\,r^{-(d+2\beta_0-1)}\mathcal{J}_{0}(u,B_r) + \mathcal{H}_r(u),
\end{equation}
where the boundary term driving homogeneity is
\[
    \mathcal{H}_r(u) \coloneqq \frac{1}{2}\,r^{-(d+2(\beta_0-1))}
    \int_{\partial B_r}\left(\partial_\nu u - \beta_0 r^{-1}u\right)^2\,d\mathcal{H}^{d-1}.
\]
Now consider the $\beta_0$-homogeneous extension of $u$ from $\partial B_r$ into $B_r$, defined by
\[
    w(x) \coloneqq \left(\frac{|x|}{r}\right)^{\beta_0} u\left(r\frac{x}{|x|}\right)
    \quad \text{for } x\neq 0,
    \qquad
    w(0)\coloneqq 0.
\]
Since $w$ is $\beta_0$-homogeneous with respect to the origin, its Weiss functional is constant. Therefore, using \eqref{eq:derivative-formula}, we have
\[
    0 = \frac{d}{ds}\mathcal{A}_{w,0}(s)\Big|_{s=r}
    = \mathcal{B}(w) -(d+2(\beta_0-1))\,r^{-(d+2\beta_0-1)}\mathcal{J}_{0}(w,B_r)+\mathcal{H}_r(w).
\]
Moreover, by homogeneity, we have $\partial_\nu w=\beta_0 r^{-1}w$ on $\partial B_r$, hence $\mathcal{H}_r(w)=0$, and consequently
\[
    \mathcal{B}(w)=(d+2(\beta_0-1))\,r^{-(d+2\beta_0-1)}\mathcal{J}_{0}(w,B_r).
\]
Since $u=w$ on $\partial B_r$, it follows that $\mathcal{B}(u)=\mathcal{B}(w)$, and thus
\[
    \mathcal{B}(u)=(d+2(\beta_0-1))\,r^{-(d+2\beta_0-1)}\mathcal{J}_{0}(w,B_r).
\]
Finally, we split the frozen energy of $w$ as
\[
    \mathcal{J}_{0}(w,B_r)=\mathcal{J}(w,B_r)+\bigl(\mathcal{J}_{0}(w,B_r)-\mathcal{J}(w,B_r)\bigr),
\]
and obtain
\[
    \mathcal{B}(u)=(d+2(\beta_0-1))\,r^{-(d+2\beta_0-1)} \Big(\mathcal{J}(w,B_r)+\mathcal{J}_{0}(w,B_r)-\mathcal{J}(w,B_r)\Big).
\]
Since $u$ is an almost minimizer and $w$ is an admissible competitor, we have
\[
    \mathcal{J}(u,B_r)\le (1+\varrho(r))\,\mathcal{J}(w,B_r).
\]
Using that $\mathcal{J}\ge 0$ and the elementary inequality
\[
    \frac{1}{1+t}\ge 1-t \quad\text{for every } t>0,
\]
we infer
\[
    (1-\varrho(r))\,\mathcal{J}(u,B_r)\le \mathcal{J}(w,B_r).
\]
Consequently,
\[
    \mathcal{B}(u)
    \ge \kappa_r(1-\varrho(r))\,\mathcal{J}(u,B_r)
    + \kappa_r\bigl(\mathcal{J}_0(w,B_r)-\mathcal{J}(w,B_r)\bigr),
\]
where
\[
    \kappa_r \coloneqq (d+2(\beta_0-1))\,r^{-(d+2\beta_0-1)}.
\]
Writing 
\[
	\mathcal{J}(u,B_r) = \mathcal{J}_0(u,B_r) + \frac{1}{2}\int_{B_r} \nabla u \cdot (A(x) - I)\nabla u + \int_{B_r}\delta(x)u^{\gamma(x)} - \delta(0)u^{\gamma_0},
\]
we get
\begin{align*}
\mathcal{B}(u) \ge {}&
\kappa_r\,\mathcal{J}_0(u,B_r)
-\kappa_r\,\varrho(r)\,\mathcal{J}(u,B_r)
\\
&+\kappa_r \int_{B_r}\Big[
\delta(x)\big(u^{\gamma(x)}-w^{\gamma(x)}\big)
+\delta(0)\big(w^{\gamma_0}-u^{\gamma_0}\big)
\Big]\,dx
\\
&+\frac{1}{2}\kappa_r \int_{B_r}\Big[
\nabla u\cdot(A(x)-I)\nabla u
-\nabla w\cdot(A(x)-I)\nabla w
\Big]\,dx.
\end{align*}
Using this to estimate the derivative of $\mathcal{A}_{u,0}(r)$ and reorganizing terms, we obtain
\begin{align*}
    \frac{d}{dr} \mathcal{A}_{u,0}(r) & \geq \mathcal{H}_r(u) - \kappa_r \varrho(r)\mathcal{J}(u,B_r) + \kappa_r \int_{B_r}\left(\delta(0)w^{\gamma_0} - \delta(x)w^{\gamma(x)} \right)\,dx\\
    & \quad - \kappa_r \int_{B_r}\Big(\delta(0)u^{\gamma_0} - \delta(x)u^{\gamma(x)} \Big)\,dx\\
    & \quad + \frac{1}{2}\kappa_r \bigg(\int_{B_r} \nabla u \cdot (A(x) - I)\nabla u - \int_{B_r} \nabla w \cdot (A(x) - I)\nabla w\bigg).
\end{align*}
We now estimate the error terms on the right-hand side. By \eqref{assumptions: growth for min},
\begin{align*}
\mathcal{J}(u,B_r)
&\coloneqq \int_{B_r}\bigg(\frac12\,\langle A(x)\nabla u,\nabla u\rangle + \delta(x)\,u^{\gamma(x)}\bigg)\,dx \\
&\le C\,\mathcal{H}^d(B_r)\bigl(r^{\beta_0\gamma_0}+r^{\beta_0\min_{B_r}\gamma}\bigr)
\lesssim r^{d+\beta_0\gamma_0}.
\end{align*}
In the last step we used the continuity of $\gamma$ to compare $r^{\beta_0\min_{B_r}\gamma}$ with
$r^{\beta_0\gamma_0}$ (see \cite[Proof of Theorem 4.1]{ASTU1} for the detailed argument). Since
$\beta_0\gamma_0=2(\beta_0-1)$, we obtain
\[
    r^{-(d+2\beta_0-1)}\,\varrho(r)\,\mathcal{J}(u,B_r)\lesssim \frac{\varrho(r)}{r}.
\]
Next, note that
\[
    \underbrace{\delta(0)u^{\gamma_0}-\delta(x)u^{\gamma(x)}}_{\mathcal{E}_{u,0}(x)}
    = (\delta(0)-\delta(x))u^{\gamma_0} + \delta(x)\bigl(u^{\gamma_0}-u^{\gamma(x)}\bigr).
\]
Thus, for $x\in B_r$, using again \eqref{assumptions: growth for min} we estimate
\[
    |\mathcal{E}_{u,0}(x)| \le C\bigl(\omega_\delta(r)+\omega_\gamma(r)|\ln r|\bigr)\,r^{\beta_0\gamma_0},
\]
where we used the mean value theorem applied to the function $s \mapsto t^s$. Likewise,
\[
    |\mathcal{E}_{w,0}(x)| \le C\bigl(\omega_\delta(r)+\omega_\gamma(r)|\ln r|\bigr)\,r^{\beta_0\gamma_0}.
\]
Moreover, by the continuity assumption on $A$ and \eqref{assumptions: growth for min} once more, we have
\[
	\frac{1}{2}\kappa_r \left(\int_{B_r} \nabla u \cdot (A(x) - I)\nabla u - \int_{B_r} \nabla w \cdot (A(x) - I)\nabla w\right) \geq -C \frac{\omega_{A}(r)}{r}.
\]
Therefore,
\[
    \frac{d}{dr} \mathcal{A}_{u,0}(r) \geq \mathcal{H}_r(u) - C\underbrace{\frac{\left(\omega_A(r) + \varrho(r) + \omega_\delta(r) + \omega_\gamma(r)|\ln(r)|\right)}{r}}_{\coloneqq g(r)}.
\]
The assumptions on $\delta$, $\gamma$, $A$ and on the gauge function implies that $g \in L^1(0,2)$.
\end{proof}

\begin{remark}
The assumption $g\in L^1$ is needed to guarantee that the limit
\[
    \lim_{r\to 0^+}\mathcal{A}_{u,z_0}(r)
\]
exists. Indeed, if the derivative satisfies only the borderline bound
\[
    f'(r)\gtrsim -r^{-1},
\]
then $f(r)+\ln r$ has a one-sided limit as $r\to 0^+$, but $f(r)$ itself may still diverge to $+\infty$.
\end{remark}

As a corollary of the Weiss-type almost monotonicity formula, every blow-up limit of a local minimizer of \eqref{def:main functional} (that is, in the case $\varrho\equiv 0$) is homogeneous.

\begin{corollary}
Let $u \geq 0$ be a local minimizer to the functional \eqref{def:main functional} in $B_2$. Assume further that
\[
	\omega_A(\cdot) + \omega_\delta(\cdot) + |\ln(\cdot)|\omega_\gamma(\cdot) \in C^{0,dini} ([0,2]).
\]
The blow-up of $u$ at $x_0 \in F(u)\cap B_{1}$ is $\beta(x_0)$-homogeneous with respect to $x_0$.
\end{corollary}

\begin{proof}
We may assume $x_0=0$. Since $A_0\coloneqq A(0)$ is uniformly elliptic, it is invertible and admits a (positive definite) square root. Let $L\coloneqq A_0^{1/2}$, so that
\[
    L^{-1}A_0L^{-T}=I_d.
\]
Define the linear change of variables and the rescaled function
\[
    v(x)\coloneqq u(Lx).
\]
A standard change-of-variables argument shows that $v$ is a local minimizer of the transformed functional
\[
    \int \bigg(\frac12\langle \widetilde A(x)\nabla w,\nabla w\rangle+\widetilde\delta(x)\,w^{\widetilde\gamma(x)}\bigg)\,dx,
\]
again in $B_2$, where
\[
    \widetilde A(x)\coloneqq L^{-1}A(Lx)L^{-T},\qquad
    \widetilde\delta(x)\coloneqq \delta(Lx),\qquad
    \widetilde\gamma(x)\coloneqq \gamma(Lx).
\]
In particular, $\widetilde A(0)=I_d$, and the coefficients $(\widetilde A,\widetilde\delta,\widetilde\gamma)$ satisfy the same structural assumptions as $(A,\delta,\gamma)$ in Theorem~\ref{thm:almost-monotonicity}.

Consider the blow-up family at the origin,
\[
    v_r(x)\coloneqq r^{-\beta(0)}v(rx),
\qquad r\in(0,1).
\]
By compactness, along a sequence $r_j\downarrow 0$ we have $v_{r_j}\to v_0$ locally in $\mathbb{R}^n$, where
$v_0\not\equiv 0$ is a local minimizer of the frozen Alt--Phillips functional $\mathcal{J}_0$ at the origin; see Lemma~\ref{lemma:compactness}.

Fix $s>0$. Using the scaling of the Weiss functional and the convergence of $v_{r_j}$, we obtain
\[
    \mathcal{A}_{v_0,0}(s)
    =\lim_{j\to\infty}\mathcal{A}_{v_{r_j},0}(s)
    =\lim_{j\to\infty}\mathcal{A}_{v,0}(sr_j)
    =\mathcal{A}_{v,0}(0^+),
\]
where the last equality follows from Theorem~\ref{thm:almost-monotonicity}. Hence
$\mathcal{A}_{v_0,0}(s)$ is constant in $s>0$.

At this stage, one may invoke the classical classification of homogeneous minimizers for the constant-parameter Alt--Phillips problem. Alternatively, revisiting the computation leading to \eqref{eq:derivative-formula}, the constancy of $\mathcal{A}_{v_0,0}$ yields, for every $s>0$,
\[
    \mathcal{B}(v_0)
    -\bigl(d+2(\beta_0-1)\bigr)s^{-(d+2\beta_0-1)}\mathcal{J}_0(v_0,B_s)
    +\mathcal{H}_s(v_0)=0.
\]
Moreover, since $v_0$ minimizes $\mathcal{J}_0$, the same comparison argument used in the proof of
Theorem~\ref{thm:almost-monotonicity} gives
\[
    \mathcal{B}(v_0)
    -\bigl(d+2(\beta_0-1)\bigr)s^{-(d+2\beta_0-1)}\mathcal{J}_0(v_0,B_s)
    \ge -g(s).
\]
In the constant-parameter case we have $g\equiv 0$, and therefore $\mathcal{H}_s(v_0)=0$. Recalling the definition of $\mathcal{H}_s$, we conclude that for every $s>0$,
\[
    \int_{\partial B_s}\bigl(\partial_\nu v_0-\beta_0 s^{-1}v_0\bigr)^2\,d\mathcal{H}^{d-1}=0,
\]
hence
\[
    \partial_\nu v_0=\beta_0 s^{-1}v_0 \quad \text{on }\partial B_s.
\]
An ODE argument along rays then gives
\[
    v_0(\lambda x)=\lambda^{\beta_0}v_0(x)
\quad \text{for all }\lambda>0\text{ and }x\in\mathbb{R}^n.
\]
\end{proof}

\begin{remark}
    
\end{remark}

\section{The Alt--Philips problem with varying coefficients}\label{sct:Alt-philips-varying}

In this section, we revisit fine regularity properties of the free boundary of nonnegative local minimizers of the so-called `Alt--Philips problem with varying coefficients'
\begin{equation}\label{functional:alt_philips-varying}
	\mathcal{J}_\delta^\gamma(u) \coloneqq \int \left( \frac{|\nabla u|^2}{2} + \delta(x)\,u^{\gamma(x)} \right)\,dx.
\end{equation}
Local minimizers to this functional were studied in detail in the recent work \cite{ASTU1}, in which the main result is that if
\[
    \delta, \gamma \in W^{1,q}_{loc} \quad \text{for} \quad q>d,
\]
then the free boundary $F(u)\cap B_1$ is locally a $C^{1,\gamma}$ surface up to a negligible set $\Sigma(u)$ of Hausdorff dimension less than or equal to $d-3$. Moreover, if 
\[
    \delta, \gamma \in W^{2,\infty}_{loc},
\]
then $\mathcal{H}^{d-1}(F(u)\cap B_r) \lesssim r^{d-1}$.

At first glance, both assumptions may appear to impose an excessively high degree of regularity; however, they arise naturally from classical methods. The standard derivation of the Weiss monotonicity formula requires differentiating the functional once, while the derivation of Hausdorff measure estimates for the Alt--Phillips problem requires differentiating the equation twice. In this framework, both assumptions are essentially sharp, as they are needed to control the error terms generated by the derivatives of \(\gamma\) and \(\delta\). The reasoning here in this section allows us to significantly reduce these regularity assumptions to
\begin{equation}\label{assump:opt-regularity-exponents}
     \delta \in C^{\mathrm{dini}}_{loc} \qquad \text{and} \qquad \gamma \in C^{\log-dini}_{loc}.
\end{equation}

The main building blocks to that end are that
\begin{itemize}
    \item[(I)] $\mathcal{R}(u)$ is locally open in $F(u)$;
    \item[(II)] Flatness implies regularity;
    \item[(III)] Dimension reduction argument,
\end{itemize}
which we prove in the forthcoming subsections. Item~\((I)\) is essentially a consequence of the Weiss monotonicity formula. Item~\((II)\) was established, in essence, in \cite[Section~7]{ASTU1}. In that work, we assumed that both \(\delta\) and \(\gamma\) were H\"older continuous, as a consequence of the \(W^{1,q}\) assumption with \(q>d\). However, the argument only requires that, after scaling, both \(\delta\) and \(\gamma\) become almost constant. We also emphasize that, within the positivity set, the solution is classical in view of the continuity assumptions on \(\delta\) and \(\gamma\), together with Kovats' borderline regularity results \cite{JK1}. Item~\((III)\) likewise follows from the Weiss monotonicity formula and the translation invariance of homogeneous blow-ups. We summarize these conclusions in the following theorem.

\begin{theorem}\label{thm:main-thm-varying-AP}
Let $u \geq 0$ be a local minimizer to the energy-functional \eqref{functional:alt_philips-varying} in $B_2$. Assume further that Assumption \eqref{assump:opt-regularity-exponents} is in force. Then, $\mathcal{R}(u)$ is locally $C^{1,\alpha}$, and
\[
	\mathrm{dim}_\mathcal{H}\left(\Sigma(u)\right) \le d-3.
\]
Moreover, $\mathcal{H}^{d-1}(F(u)\cap B_r) \lesssim r^{d-1}$, for $r \leq 1$.
\end{theorem}

We divide the proof of this theorem into several parts to make the argument more easily followable.

\subsection{The regular set is locally open}

Let $x_0 \in \mathcal{R}(u)$. By definition, the family $(u_r)_{r>0}$, defined by
\[
    u_r(x) \coloneqq r^{-\beta(x_0)}u(x_0 + rx)
\]
converges, up to subsequence, to the trivial cone $u_{x_0}(x) \coloneqq \vartheta(x_0)(x \cdot \nu)_+^{\beta(x_0)}$, for some vector $\nu \in \partial B_1$. The constant $\vartheta(x_0)$ is chosen so that $u_{x_0}$ is a local minimizer to the Alt--Philips functional frozen at $x_0$
\[
    \int \left( \frac{|\nabla u|^2}{2} + \delta(x_0)\,u^{\gamma(x_0)} \right)\,dx.
\]
The precise value of the constant $\vartheta(x_0)$ is
\[
	\vartheta(x_0) \coloneqq \left(\frac{(\beta(x_0)-1)\beta(x_0)}{\gamma(x_0)\, \delta(x_0)}\right)^{\frac{1}{\gamma(x_0)-2}}.
\]
The Weiss energy of those functions is constant and depends on the values of $\gamma$ and $\delta$ applied to $x_0$. This is a direct computation based on its explicit form.

\begin{lemma}\label{lemma: weiss energy of trivial}
Let $x_0 \in \mathbb{R}^d$, $u_{x_0}(x) \coloneqq \vartheta(x_0)(x \cdot \nu)_+^{\beta(x_0)}$ with $\nu \in \partial B_1$, and $\mathcal{A}_{u,0}(\cdot)$ be the Weiss energy defined in Theorem \ref{thm:almost-monotonicity}. Then, there is $\mathcal{V}$ depending on the values of $\gamma$, $\delta$ applied on $x_0$, and dimension such that
\[
    \mathcal{A}_{u_{x_0},0}(r) = \mathcal{V}(x_0) \quad \text{for all} \quad r>0.
\]
Moreover, $\mathcal{V}$, as a function of $x_0$, is continuous and its modulus of continuity depends on those of $\gamma$ and $\delta$.
\end{lemma}

Now we prove an energy gap for the Weiss energy. The proof is essentially the same as that for the constant parameter, but we include it here for completeness.

\begin{proposition}\label{prop:energy-gap}
Let $u \geq 0$ be a local minimizer to the energy-functional \eqref{functional:alt_philips-varying} in $B_2$, $U$ be a blow up at a point $z_0 \in F(u)$ that is a nontrivial minimal cone, and $\mathcal{A}_{U,0}$ its Weiss energy. Assume further that Assumption \eqref{assump:opt-regularity-exponents} is in force. Then, there exists a universal constant $\eta>0$ such that
\begin{equation}\label{eq:gap-at-0}
    \mathcal{A}_{U,0}\ge \mathcal{V}(z_0)+\eta.
\end{equation}
\end{proposition}

\begin{proof}
For simplicity, we assume $z_0 = 0$ and write $\mathcal{A}_{U} \equiv \mathcal{A}_{U,0}$. First, we prove the strict inequality $\mathcal{A}_{U}>\mathcal{V}(0)$.
Assume, toward a contradiction, that
\[
    \mathcal{A}_{U}\le \mathcal{V}(0).
\]
Let $X_0 \in \mathcal{R}(U)$ be a regular free boundary point, and consider the blow-up sequence
\[
    U_r(X)\coloneqq r^{-\beta(0)}U(X_0+rX).
\]
Taking a subsequence if necessary, it converges to a half-space solution of the form
\[
    U_0(X) = \vartheta(0)\,(X\cdot \nu)_+^{\beta(0)} \quad \text{for some} \quad \nu \in \partial B_1.
\]
In particular, by Lemma \ref{lemma:compactness}, $U_0$ is a local minimizer of the functional
\begin{equation}\label{eq:func-at-0}
	\int \left(\frac{|\nabla v|^2}{2} + \delta(0) v^{\gamma(0)} \right)\,dx,
\end{equation}
and the Weiss energy of $U_0$ around $0$ equals $\mathcal{V}(0)$. Since $U$ is $\beta(0)$-homogeneous with respect to the origin, we have
\[
    U_r(X)=U\bigl(X+r^{-1}X_0\bigr)\quad \text{for all }X\in \mathbb{R}^d,
\]
and therefore 
\begin{align*}
	\mathcal{A}_{U,X_0}(r) & = \mathcal{A}_{U_r,r^{-1}X_0}(1)\\
	& = \mathcal{J}_{0}(U_r,B_1(r^{-1}X_0)) - \frac{1}{2}\beta(0)\int_{\partial B_1(r^{-1}X_0)}U_r^2\, d\mathcal{H}^{d-1}.
\end{align*}
By the convergence of $U_r$ to $U_0$ and the contradiction assumption, we obtain
\[
	\lim_{r \to \infty}\mathcal{A}_{U,X_0}(r) = \mathcal{A}_{U,0}(1) \leq \mathcal{V}(0).
\]
On the other hand, since $X_0$ is a regular point, we have
\[
    \lim_{r\to 0}\mathcal{A}_{U,X_0}(r) = \mathcal{V}(0).
\]
Since $r\mapsto \mathcal{A}_{U,X_0}(r)$ is monotone nondecreasing, it must be constant and equal to $\mathcal{V}(0)$ for all $r>0$. This implies that $U$ is $\beta(0)$-homogeneous with respect to $X_0$. 
Hence $U$ is a cone with vertex at $X_0$ as well as at $0$, and therefore must be a half-space solution, contradicting the assumption that $U$ is nontrivial. This can be proved by induction in dimension, as in \cite[Proof of Proposition~8.1]{W3}. This proves $\mathcal{A}_{U}>\mathcal{V}(0)$.

We now show that the strict inequality is quantitative, i.e.\ \eqref{eq:gap-at-0} holds for some universal $\eta>0$. Assume not. Then there exist nontrivial minimizing cones $U_k$ for functional \eqref{eq:func-at-0} and numbers $\eta_k\downarrow 0$ such that
\[
    \mathcal{A}_{U_k} < \mathcal{V}(0) + \eta_k.
\]
By the preceding argument, we also have $\mathcal{A}_{U_k}\ge \mathcal{V}(0)$, hence
\[
    \mathcal{A}_{U_k}\to \mathcal{V}(0) \quad \text{as} \quad k \to \infty.
\]
Using compactness of minimizers, we may assume
\[
    U_k\to U_\infty \quad \text{locally uniformly in }\mathbb{R}^d,
\]
where $U_\infty$ is a minimizing cone. By the lower semicontinuity of the Weiss energy and the limit above,
\[
    \mathcal{A}_{U_\infty}=\mathcal{V}(0),
\]
and so, in view of the preceding argument, $U_\infty$ must be the trivial cone. Since $U_k\to U_\infty$, for $k$ sufficiently large the free boundaries $F(U_k)$ are sufficiently flat in $B_1$. By the ``flatness implies regularity'' theorem for the constant-parameter Alt--Phillips functional, each $F(U_k)$ is $C^{1}$ in $B_1$ for $k$ large. Because $U_k$ is homogeneous with respect to $0$, the set $F(U_k)$ is a cone:
\[
    x\in F(U_k)\ \Longrightarrow\ \lambda x\in F(U_k)\qquad \forall \lambda>0.
\]
In particular, $F(U_k)$ is a $C^{1}$ conical hypersurface with vertex at $0$, and hence it must be a hyperplane. 
Therefore $\{U_k>0\}$ is a half-space.

Finally, in a half-space the only nonnegative $\beta(0)$-homogeneous minimizer of the frozen Alt--Phillips functional is the half-space solution
\[
    U_k(X)=\vartheta(0)\,(X\cdot \nu)_+^{\beta(0)}.
\]
This contradicts the choice of $U_k$ as nontrivial minimizing cones. 
\end{proof}

As a consequence of the energy gap, we obtain that $\mathcal{R}(u)$ is locally open in $F(u)$.

\begin{corollary}
Let $u \geq 0$ be a local minimizer to the energy-functional \eqref{functional:alt_philips-varying} in $B_2$. Assume further that Assumption \eqref{assump:opt-regularity-exponents} is in force. Then, $\mathcal{R}(u)$ is locally open in $F(u)$.
\end{corollary}

\begin{proof}
Let $x_0 \in \mathcal{R}(u)$. We shall prove that there exists $r>0$ such that every point in $B_r(x_0) \cap F(u) \subset \mathcal{R}(u)$. Arguing by contradiction, assume that no such radius exists. Then, we would be able to find a sequence $x_j \to x_0$, with $x_j \in \Sigma(u)$. Let $\eta^\prime>0$ be a small parameter to be chosen later. Since $x_0 \in \mathcal{R}(u)$, we obtain
\[
    \lim_{r \to 0} \mathcal{A}_{u,x_0}(r) = \mathcal{V}(x_0).
\]
We then let $r_0>0$ small such that
\[
    \mathcal{A}_{u,x_0}(r_0) \le \mathcal{V}(x_0) + \eta^\prime.
\]
For $j$ large enough, we obtain
\begin{equation}\label{eq:Weiss-modulus-collapsed}
\big|\mathcal{A}_{u,x_0}(r_0)-\mathcal{A}_{u,x_j}(r_0)\big|
\ \le\
C'\Big[
\tfrac{d_j}{r_0}\;+\;\omega_\delta(d_j)\;+\;\omega_\gamma(d_j)\,(1+|\log r_0|)
\Big],
\end{equation}
with $d_j = |x_j - x_0| \le r_0/2$. Therefore,
\[
   \mathcal{A}_{u,x_j}(r_0) \le \mathcal{V}(x_0) + \eta^\prime + C'\Big[
\tfrac{d_j}{r_0}\;+\;\omega_\delta(d_j)\;+\;\omega_\gamma(d_j)\,(1+|\log r_0|)
\Big].
\]
In particular, for $j$ large enough, \(\mathcal{A}_{u,x_j}(r_0) \le \mathcal{V}(x_0) + 2\eta^\prime\). By the almost monotonicity formula, Theorem \ref{thm:almost-monotonicity}, we have that there exists $g \in L^1$ such that
\[
    \frac{d}{dr}\mathcal{A}_{u,x_j}(r) \geq -g(r),
\]
from which follows
\[
    \mathcal{A}_{u,x_j}(r) \le \mathcal{A}_{u,x_j}(r_0) + f(r_0) - f(r),
\]
where $f$ is the anti derivative of $g$. Then, we obtain
\[
    \mathcal{A}_{u,x_j}(r) \le \mathcal{A}_{u,x_j}(r_0) + \int_r^{r_0}g(s)\,ds.
\]
Putting all together, and passing to the limit as $r \to 0$, we get
\begin{align*}
    \mathcal{A}_{U_j,0} &\le \mathcal{V}(x_0) + 2\eta^\prime + \int_0^{r_0}g(s)\,ds\\
    & \le \mathcal{V}(x_j) + \omega_{\mathcal{V}}(d_j) + 2\eta^\prime + \int_0^{r_0}g(s)\,ds,
\end{align*}
where $U_j$ is the blow-up limit at $x_j \in \Sigma(u)$. Next, we fix $\eta^\prime > 0$ sufficiently small and then choose $r_0 > 0$ sufficiently small so that
\[
\mathcal{A}_{U_j,0} \leq \mathcal{V}(x_j) + \frac{\eta}{2},
\]
where $\eta$ is the constant appearing in Proposition \ref{prop:energy-gap}. This is incompatible with \eqref{eq:gap-at-0}, yielding a contradiction.
\end{proof}

\subsection{Dimension reduction}

We show that the dimension of the singular set is less than or equal to $d-3$. We first prove the following proposition, which relates the size of the singular set of the minimizers to the blow-ups. This is an adaptation of the argument in \cite{DSS0}.

\begin{proposition}\label{lemma:singular-set-blow_up-to-minimizer}
Let $u \geq 0$ be a local minimizer to the energy-functional \eqref{functional:alt_philips-varying} in $B_2$. Assume further that Assumption \eqref{assump:opt-regularity-exponents} is in force. For $x_0\in \Sigma(u)$, let $\mathcal{S}_u(x_0)$ denote the set of blow-ups of $u$ at $x_0$, and set
\[
    \mathcal{S}_u \coloneqq \bigcup_{x_0\in\Sigma(u)} \mathcal{S}_u(x_0).
\]
If there exists $s>0$ such that
\[
    \mathcal H^s\bigl(\Sigma(U)\bigr)=0 \quad \text{for every} \quad U\in \mathcal{S}_u,
\]
then
\[
    \mathcal H^s(\Sigma(u))=0.
\]
\end{proposition}

\begin{proof}
We first state a local covering improvement property at singular points. Its proof is postponed to the end.

\medskip
\noindent\textbf{Claim.} For each \(x \in \Sigma(u)\), there exists \(d_x > 0\) with the following property: whenever \(0 < \mu \le d_x\) and \(D \subset \Sigma(u) \cap B_\mu(x)\), there exists a finite collection of balls \(\{B_{r_i}(x_i)\}_{i=1}^N\), centered at points \(x_i \in D\), such that
\begin{equation}\label{eq:local_half_cover}
D\subset \bigcup_{i=1}^{N} B_{r_i}(x_i),
\qquad
\sum_{i=1}^{N} r_i^{\,s}\le \frac{\mu^s}{2}.
\end{equation}

Assuming the claim, for each $m \in \mathbb{N}$, define the set
\[
	D_m \coloneqq \{x \in \Sigma(u) : d_x \ge m^{-1}\}.
\]
We prove that $\mathcal{H}^s(D_m) = 0$ for every $m \in \mathbb{N}$. It is enough to show that
\[
	\mathcal{H}^s\bigl(D_m \cap B_{m^{-1}}(x)\bigr) = 0
\]
for every $x \in \Sigma(u)$. To this end, we construct inductively, for each $k \in \mathbb{N}$, a covering
\[
	\{B_{r_i}(x_i)\}_{i=1}^{\mathcal{N}_k}
\]
of $D_m \cap B_{m^{-1}}(x)$ such that
\begin{equation}\label{eq:decay-covering-measure}
\sum_{i=1}^{\mathcal{N}_k} r_i^s \lesssim 2^{-k}.
\end{equation}

Fix $m \in \mathbb{N}$ and $x \in \Sigma(u)$. For $k=1$, we apply the claim with $\mu = m^{-1} \le d_x$ and
\[
	D = D_m \cap B_{m^{-1}}(x),
\]
obtaining a covering $\{B_{r_i}(x_i)\}_{i=1}^{\mathcal{N}_1}$ for which \eqref{eq:local_half_cover} holds. This establishes the case $k=1$.

Now assume that the statement holds for some $k \in \mathbb{N}$, namely, that there exists a covering
\[
	\{B_{r_i}(x_i)\}_{i=1}^{\mathcal{N}_k}
\]
of $D_m \cap B_{m^{-1}}(x)$ satisfying \eqref{eq:decay-covering-measure}. For each $i \in \{1,\dots,\mathcal{N}_k\}$, we apply the claim with
\[
	D = D_m \cap B_{r_i}(x_i),
\]
and obtain a covering
\[
	\mathcal{C}_{i,k} \coloneqq \{B_{r_j}(x_j)\}_{j=1}^{\mathcal{N}_k^i}
\]
such that
\[
	\sum_{j=1}^{\mathcal{N}_k^i} r_j^s \le \frac{r_i^s}{2}.
\]
Taking the union of all these families, we obtain a covering of $D_m \cap B_{m^{-1}}(x)$, and therefore
\[
	\sum_{i=1}^{\mathcal{N}_k} \sum_{j=1}^{\mathcal{N}_k^i} r_j^s
	\le \frac{1}{2} \sum_{i=1}^{\mathcal{N}_k} r_i^s
	\lesssim 2^{-(k+1)}.
\]
Thus \eqref{eq:decay-covering-measure} holds at level $k+1$, completing the induction.

It follows that the set $D_m \cap B_{m^{-1}}(x)$ admits a covering by balls whose $s$-dimensional size is arbitrarily small. Hence
\[
	\mathcal{H}^s\bigl(D_m \cap B_{m^{-1}}(x)\bigr) = 0.
\]

Finally, since
\[
	\Sigma(u) = \bigcup_{m=1}^\infty D_m,
\]
countable subadditivity of $\mathcal{H}^s$ yields
\[
	\mathcal{H}^s(\Sigma(u)) = 0,
\]
and the proposition is proved.

Now we return to the proof of the claim. Assume, by contradiction, that there exists $y\in\Sigma(u)$ and a sequence $\mu_k\downarrow 0$ such that~\eqref{eq:local_half_cover} fails for some choice of
$D \subset \Sigma(u)\cap B_{\mu_k}(y)$.

Consider the family $(u_k)_{k \in \mathbb{N}}$ defined as
\[
    u_k(X)\coloneqq \mu_k^{-\beta(y)}\,u\bigl(y+\mu_k X\bigr).
\]
By compactness, after passing to a subsequence, we have $u_k\to U_y$ locally uniformly in $\mathbb R^n$, where $U_y\in \mathcal{S}_u(y)$. By the hypothesis of the lemma,
\[
    \mathcal H^s\bigl(\Sigma(U_y)\bigr)=0.
\]
Hence, there exists a finite cover $\{B_{\rho_i}(X_i)\}_{i=1}^{\mathcal N_1}$ of $\Sigma(U_y)\cap B_1$, such that
\begin{equation}\label{eq:cover_blowup}
\Sigma(U_y)\cap B_1 \subset \bigcup_{i=1}^{\mathcal N_1} B_{\rho_i}(X_i),
\quad \text{with} \quad
\sum_{i=1}^{N_1} \rho_i^{\,s}\le \frac12.
\end{equation}
Set
\[
    \mathcal K \coloneqq \bigl(F(U_y)\cap B_1\bigr)\setminus \bigcup_{i=1}^{N_1} B_{\rho_i}(X_i).
\]
Then $\mathcal K$ is compact and, by construction, $\mathcal{K} \subset \mathcal{R}(U_y)$. As a consequence, for each $Z \in K$ there exist $r_Z>0$ such that $U_y$ satisfies the flatness assumption in $B_{r_Z}(Z)$, meaning that there exists $\nu \in \partial B_1$ such that
\[
    \vartheta(y)((X-Z)\cdot \nu - \varepsilon)_+^{\beta(y)} \le U_y(X) \le \vartheta(y)((X-Z)\cdot \nu + \varepsilon)_+^{\beta(y)}, \quad \text{for} \quad X \in B_{r_Z}(Z).
\]
By the convergence of $u_k$ to $U_y$, up to adjusting parameters, $u_k$ satisfies the flatness assumption as well. Therefore, by the flatness implies regularity result from \cite{ASTU1}, $F(u_k)\cap B_{r_Z/2}(Z) \subset \mathcal{R}(u_k)$, and in particular
\[
    \Sigma(u_k)\cap B_{r_Z/2}(Z)=\varnothing\qquad\text{for all sufficiently large }k.
\]
Since $Z \in \mathcal{K}$ is arbitrary, we conclude that,
\begin{equation}\label{eq:singular_inclusion_scaled}
\Sigma(u_k)\cap \mathcal{K} = \varnothing,
\qquad\text{hence}\qquad
\Sigma(u_k)\cap B_1 \subset \bigcup_{i=1}^{\mathcal N_1} B_{\rho_i/2}(X_i).
\end{equation}
Scaling~\eqref{eq:singular_inclusion_scaled} back to the original variables gives, for all large $k$,
\[
    \Sigma(u)\cap B_{\mu_k}(y)\subset \bigcup_{i=1}^{\mathcal N_1} B_{\rho_i\mu_k/2}\bigl(y+\mu_k X_i\bigr), \qquad
\sum_{i=1}^{\mathcal N_1}(\rho_i\mu_k/2)^s \le \frac{\mu_k^s}{2},
\]
which contradicts \eqref{eq:local_half_cover}. This proves the claim.
\end{proof}

Now we estimate the size of the singular set.

\begin{theorem}\label{thm:dimension-singular-set}
Let $u \geq 0$ be a local minimizer to the energy-functional \eqref{functional:alt_philips-varying} in $B_2$. Assume further that Assumption \eqref{assump:opt-regularity-exponents} is in force. It follows that
\[
	\mathrm{dim}_\mathcal{H}\left(\Sigma(u)\right) \le d-3.
\]
\end{theorem}

\begin{proof}
First, recall the definition of dimension:
\[
    \mathrm{dim}_\mathcal{H}\left(\Sigma(u)\right) \coloneqq \inf \left\{s \geq 0 \colon \mathcal{H}^s(\Sigma(u)) = 0\right\}.
\]
Define 
\[
	\mathcal{S}_u \coloneqq \left\{U_{x_0}\colon \mathbb{R}^n \to \mathbb{R} \colon U_{x_0} \text{ is a blow-up of } u \text{ at } x_0 \in \Sigma(u)\right\}.
\]
By the Weiss (almost) monotonicity formula, it follows that $U_{x_0} \in \mathcal{S}_u$ is $\beta(x_0)$-homogeneous and is a minimizer of the Alt--Philips functional frozen at $x_0$, for any $x_0 \in \Sigma(u)$. But we know that
\[
   \mathrm{dim}_\mathcal{H}\left(\Sigma(U_{x_0})\right) \le d-3, 
\]
meaning that 
\[
    \mathcal{H}^s\left(\Sigma(U_{x_0})\right) = 0, \quad \text{for every} \quad s > d-3, 
\]
for every $x_0 \in \Sigma(u)$. By Proposition \ref{lemma:singular-set-blow_up-to-minimizer}, it also follows that
\[
    \mathcal{H}^s\left(\Sigma(u)\right) = 0, \quad \text{for every} \quad s > d-3,
\]
from which the conclusion of the Theorem follows.
\end{proof}

\subsection{Hausdorff measure estimates}

In this subsection, we prove that the free boundary of any local minimizer of~\eqref{functional:alt_philips-varying} has finite \((d-1)\)-dimensional Hausdorff measure. The first step is a decomposition lemma: we split the free boundary into a \emph{good} part, whose \(\mathcal H^{d-1}\)-measure is directly controlled, and a \emph{bad} part, which can be covered by balls with small \((d-1)\)-content. The argument is adapted from~\cite{DSS0}.

\begin{lemma}\label{lemma : take a ball}
Let $u \geq 0$ be a local minimizer to the energy-functional \eqref{functional:alt_philips-varying} in $B_2$. Assume further that Assumption \eqref{assump:opt-regularity-exponents} is in force. There exists $\mu_0 > 0$ such that, if $\|u\|_{L^\infty(B_{4/3})} \leq 1$, then
\[
    \mathcal{H}^{d-1}\left( \left(F(u) \cap B_{\mu}(x_0)\right) \setminus \bigcup_{i=1}^{\mathcal{N}} B_{\delta_i}(x_i) \right) \lesssim \mu^{d-1},
\]
for every $0 < \mu \leq \mu_0$, every $x_0 \in F(u) \cap B_1$, and every finite family of balls $\{B_{\delta_i}(x_i)\}_{i=1}^{\mathcal{N}}$ satisfying
\[
    \sum_{i=1}^{\mathcal{N}} \delta_i^{d-1} \leq \frac{1}{4}\mu^{d-1}.
\]
\end{lemma}

\begin{proof}
Define
\[
    v_\mu(x) \coloneqq \mu^{-\beta(x_0)}u(x_0+\mu x),
\]
and observe that \(v_\mu\) is a local minimizer of a similar functional with parameters $\delta_\mu(x)$ and $\gamma(x)$, with \(\delta_\mu \to \delta(x_0)\) and \(\gamma_\mu \to \gamma(x_0)\) as \(\mu \to 0\). In this setting, it suffices to prove that
\[
    \mathcal{H}^{d-1}\left(\left(F(v_\mu)\cap B_1\right)\setminus \bigcup_{i=1}^{\mathcal N} B_{\delta_i}(x_i)\right) \lesssim 1.
\]
Assume, seeking a contradiction, that this is false. Then, we can find a sequence \(\mu_k \to 0\) such that \(v_k \coloneqq v_{\mu_k}\) satisfies
\begin{equation}\label{eq: contradiction assumption hausdorff}
    \mathcal{H}^{d-1}\left(\left(F(v_k)\cap B_1\right)\setminus \bigcup_{i=1}^{\mathcal N} B_{\delta_i}(x_i)\right) \geq k,
\end{equation}
for every finite collection of balls \(\{B_{\delta_i}(x_i)\}_{i=1}^{\mathcal N}\) such that
\[
    \sum_{i=1}^{\mathcal N} \delta_i^{d-1} \leq \frac{1}{4}.
\]
Passing to a subsequence, if necessary, we obtain \(v_k \to v_\infty\) in $C^{1,\frac{\gamma(x_0)}{2-\gamma(x_0)}}_{loc}(\mathbb{R}^d)$, where \(v_\infty\), by Lemma \ref{lemma:compactness}, is a local minimizer in \(B_2\) to the Alt--Philips functional frozen at $x_0$
\[
    \int \left(\frac{|\nabla v|^2}{2}+\delta(x_0)v^{\gamma(x_0)}\right)\,dx.
\]
From the classical theory, we know that \(\mathcal{H}^{d-1}(\Sigma(v_\infty)) = 0\). Therefore,
\[
    \Sigma(v_\infty) \cap B_1 \subset \bigcup_{i=1}^{\mathcal{N}_0} B_{\delta_i/2}(x_i), \qquad \text{with} \qquad \sum_{i=1}^{\mathcal{N}_0} \delta_i^{d-1} \leq \frac{1}{2},
\]
for some finite collection of balls \(\{B_{\delta_i/2}(x_i)\}_{i=1}^{\mathcal{N}_0}\), where \(x_i \in \Sigma(v_\infty)\). Define
\[
    S_k \coloneqq \left(F(v_k) \cap B_1\right) \setminus \bigcup_{i=1}^{\mathcal{N}_0} B_{\delta_i/2}(x_i).
\]
By the convergence of \(F(v_k)\) to \(F(v_\infty)\), it follows that, for \(k\) sufficiently large, the set \(S_k\) lies in a neighborhood of \(\mathcal{R}(v_\infty)\). Since \(\mathcal{R}(v_\infty)\) is locally smooth by the classical theory, and \(v_\infty\) satisfies the flatness condition in a sufficiently small neighborhood of each point of \(\mathcal{R}(v_\infty)\), the same is true for \(v_k\) when \(k\) is large enough. Hence, by the flatness implies regularity theorem from \cite{ASTU1}, each \(S_k\) can be locally parametrized as the graph of a \(C^{1,\alpha}\) function \(g_k\), with universally bounded \(C^{1,\alpha}\) norm. Passing to a subsequence if necessary, we then obtain that \(S_k\) converges in \(C^1\) to
\[
    \left(F(v_\infty) \cap B_1\right) \setminus \bigcup_{i=1}^{\mathcal{N}_0} B_{\delta_i/2}(x_i).
\]
It then follows from the area formula that
\[
    \mathcal{H}^{d-1}\left(\left(F(v_k) \cap B_1\right) \setminus \bigcup_{i=1}^{\mathcal{N}_0} B_{\delta_i/2}(x_i)\right) \to \mathcal{H}^{d-1}\left(\left(F(v_\infty) \cap B_1\right) \setminus \bigcup_{i=1}^{\mathcal{N}_0} B_{\delta_i/2}(x_i)\right),
\]
which contradicts \eqref{eq: contradiction assumption hausdorff} for \(k\) sufficiently large.
\end{proof}

A proper iteration of Lemma \ref{lemma : take a ball} gives the Hausdorff measure estimate for the free boundary.

\begin{theorem}\label{thm:Hdminus1_bound_freebdry}
Let $u \geq 0$ be a local minimizer to the energy-functional \eqref{functional:alt_philips-varying} in $B_2$. Assume further that Assumption \eqref{assump:opt-regularity-exponents} is in force. If $\|u\|_{L^\infty(B_{4/3})} \leq 1$, then
\[
    \mathcal H^{d-1}\bigl(F(u)\cap B_1\bigr)\lesssim 1.
\]
\end{theorem}

\begin{proof}
We argue inductively. For each integer $k\ge 0$ we construct a finite index set $\mathcal I_k \subset \mathbb{N}^k$, where $\mathbb{N}^k$ is the product of $\mathbb{N}$ $k$-times, a family of balls
\[
    \mathcal B_k \coloneqq \bigl\{B_{\delta_{\mathbf i}}(x_{\mathbf i}) \colon \mathbf i\in\mathcal I_k\bigr\},
\]
together with sets $\Gamma_{\mathbf i}\subset F(u)\cap B_{\delta_{\mathbf i}}(x_{\mathbf i})$ (for $\mathbf i\in\mathcal I_k$), such that
\begin{align}
F(u)\cap B_1
&\subset \left(\bigcup_{j=0}^{k-1}\ \bigcup_{\mathbf i\in\mathcal I_j}\Gamma_{\mathbf i}\right)\ \cup\ \bigcup_{\mathbf i\in\mathcal I_k} B_{\delta_{\mathbf i}}(x_{\mathbf i}),
\label{eq:cover_level_k}\\[2mm]
\sum_{\mathbf i\in\mathcal I_k}\delta_{\mathbf i}^{d-1}
&\le 4^{-(k+1)} \label{eq:radii_level_k}\\[2mm]
\mathcal{H}^{d-1} \left(\bigcup_{j=0}^{k-1}\ \bigcup_{\mathbf i\in\mathcal I_j}\Gamma_{\mathbf i}\right)\
&\lesssim \sum_{i=0}^{k} 4^{-(i+1)}.
\label{eq:Hausdorff_level_k}
\end{align}
For $k=0$, we apply Lemma~\ref{lemma : take a ball} to $F(u)\cap B_1$ to obtain a finite family of balls
\[
    \bigl\{B_{\delta_{i_0}}(x_{i_0})\bigr\}_{i_0=1}^{\mathcal N_0}
\]
such that
\begin{equation}\label{eq:gen0_radii}
    \sum_{i_0=1}^{\mathcal N_0}\delta_{i_0}^{d-1}\le 4^{-1},
\end{equation}
and, defining
\[
    \Gamma_0 \coloneqq \bigl(F(u)\cap B_1\bigr)\setminus \bigcup_{i_0=1}^{\mathcal N_0} B_{\delta_{i_0}}(x_{i_0}),
\]
we have
\begin{equation}\label{eq:gen0_Gamma}
\mathcal H^{d-1}(\Gamma_0)\lesssim 4^{-1},
\qquad
F(u)\cap B_1 \subset \Gamma_0\cup \bigcup_{i_0=1}^{\mathcal N_0} B_{\delta_{i_0}}(x_{i_0}).
\end{equation}
Setting $\mathcal I_0\coloneqq \{1,\dots,\mathcal N_0\}$, we have \eqref{eq:cover_level_k}--\eqref{eq:Hausdorff_level_k} hold for $k=0$, with the convention that
\[
	\Gamma_0 \coloneqq \bigcup_{j=0}^{k-1}\ \bigcup_{\mathbf i\in\mathcal I_j}\Gamma_{\mathbf i}  \quad \text{for} \quad k = 0.
\]

Assume now that $\mathcal I_k$ and $\mathcal B_k$ have been constructed for some $k\ge 1$ and that~\eqref{eq:cover_level_k}--\eqref{eq:Hausdorff_level_k} hold. Fix $\mathbf i\in\mathcal I_k$, and apply Lemma~\ref{lemma : take a ball} to $F(u)\cap B_{\delta_{\mathbf i}}(x_{\mathbf i})$. We obtain a finite family of sub-balls
\[
    \Bigl\{B_{\delta_{\mathbf i,j}}(x_{\mathbf i,j})\Bigr\}_{j=1}^{\mathcal N_{\mathbf i}},
\]
and setting
\[
    \Gamma_{\mathbf i}\coloneqq \bigl(F(u)\cap B_{\delta_{\mathbf i}}(x_{\mathbf i})\bigr)\setminus
    \bigcup_{j=1}^{\mathcal N_{\mathbf i}}B_{\delta_{\mathbf i,j}}(x_{\mathbf i,j}),
\]
satisfying the  bounds
\begin{equation}\label{eq:one_step_estimates}
    \mathcal H^{d-1}(\Gamma_{\mathbf i})\lesssim 4^{-1}\,\delta_{\mathbf i}^{d-1}, \qquad \sum_{j=1}^{\mathcal N_{\mathbf i}}\delta_{\mathbf i,j}^{d-1}\le 4^{-1}\,\delta_{\mathbf i}^{d-1}.
\end{equation}
Define the next index set and family of balls by
\[
\mathcal I_{k+1}\coloneqq \bigl\{(\mathbf i,j) \colon \mathbf i\in\mathcal I_k,\ 1\le j\le \mathcal N_{\mathbf i}\bigr\},
\quad
\mathcal B_{k+1}\coloneqq \bigl\{B_{\delta_{\mathbf i,j}}(x_{\mathbf i,j}) \colon (\mathbf i,j)\in\mathcal I_{k+1}\bigr\}.
\]
Summing~\eqref{eq:one_step_estimates} over $\mathbf i\in\mathcal I_k$ and using~\eqref{eq:radii_level_k} yields
\[
    \sum_{\mathbf j\in\mathcal I_{k+1}}\delta_{\mathbf j}^{d-1}
\le 4^{-1}\sum_{\mathbf i\in\mathcal I_k}\delta_{\mathbf i}^{d-1}
\le 4^{-1}\cdot 4^{-(k+1)}=4^{-(k+2)},
\]
from which \eqref{eq:cover_level_k}--\eqref{eq:Hausdorff_level_k} at $k+1$ follows.

To estimate $\mathcal H^{d-1}(F(u)\cap B_1)$, it is enough to show that there are sets $G_\infty$ and $E_\infty$ such that
\begin{equation}\label{eq: inclusion_FB}
    F(u) \cap B_1 \subset G_\infty \cup E_\infty,
\end{equation}
where $\mathcal{H}^{d-1}(G_\infty) \lesssim 1$, and $\mathcal{H}^{d-1}(E_\infty) = 0$, where the sets $G_\infty$ and $E_\infty$ are defined by
\[
    G_\infty = \bigcup_{k=1}^\infty G_k, \quad \text{where} \quad G_k \coloneqq \bigcup_{j=0}^{k-1}\ \bigcup_{\mathbf i\in\mathcal I_j}\Gamma_{\mathbf i},
\]
and
\[
    E_\infty = \bigcap_{k=1}^\infty E_k, \quad \text{where} \quad E_k = \bigcup_{\mathbf i\in\mathcal I_k} B_{\delta_{\mathbf i}}(x_{\mathbf i}). 
\]
Indeed, let $x \in F(u) \cap B_1$. If there is $k \in \mathbb{N}$ such that $x \in G_k$, then so it is in $G_\infty$. If there is no $k \in \mathbb{N}$ such that $x \in G_k$, then it means that $x \in E_k$ for every $k \in \mathbb{N}$, which implies that $x \in E_\infty$. This proves the inclusion \eqref{eq: inclusion_FB}. Finally, we argue that $\mathcal{H}^{d-1}(G_\infty) \lesssim 1$, and $\mathcal{H}^{d-1}(E_\infty) = 0$. While $\mathcal{H}^{d-1}(G_\infty) \lesssim 1$ follows from the fact that $G_k$ is a increasing sequence of sets with, by \eqref{eq:Hausdorff_level_k}, $\mathcal{H}^{d-1}(G_k) \lesssim 1$, to show that $\mathcal{H}^{d-1}(E_\infty) = 0$ we recall that
\[
   \mathcal{H}^{d-1}(E_\infty) = \lim_{\eta \to 0} \mathcal{H}_\eta^{d-1}(E_\infty), 
\]
where
\[
    \mathcal{H}_\eta^{d-1}(E_\infty) \coloneqq \inf \left\{\sum_{i=1}^{\infty} \mathrm{diam}(U_i)^{d-1} \colon E_\infty \subset \bigcup_{i=1}^\infty  U_i,\, \mathrm{diam}(U_i) \le \eta \right\}.
\]
But notice that $E_\infty \subset E_k$ for every $k \in \mathbb{N}$, and by \eqref{eq:radii_level_k} we have 
\[
   E_k = \bigcup_{\mathbf i\in\mathcal I_k} B_{\delta_{\mathbf i}}(x_{\mathbf i}), \quad \text{with} \quad \sum_{\mathbf i\in\mathcal I_k}\delta_{\mathbf i}^{d-1} \le 4^{-(k+1)}.
\]
Therefore,
\[
    \mathcal{H}_\eta^{d-1}(E_\infty) \subset \mathcal{H}_\eta^{d-1}(E_k) \le 4^{-(k+1)},
\]
as long as $k = k(\eta)$ is large enough so that $2\delta_{\mathbf i} \le \eta$ for ${\mathbf i} \in \mathcal{I}_k$, and from which follows that $\lim_{\eta \to 0} \mathcal{H}_\eta^{d-1}(E_\infty) = 0$.
\end{proof}

\section{Further extensions}\label{sct:some-extensions}

In this section, we bring some extensions of the almost monotonicity formula that include the one-phase Alt-Caffarelli problem which was introduced in \cite{AC}, its two-phase version, and the Alt--Philips problem with varying coefficients.

\subsection{The Alt--Caffarelli one-phase problem}

Let $Q$ and $A$ be continuous functions, bounded from below away from zero. The Alt--Caffarelli functional is defined as
\begin{equation}\label{func:AC}
    \mathcal J(u) \coloneqq \int \left( \frac{1}{2} \, \langle A(x)\nabla u, \nabla u \rangle + Q(x)\mathcal{X}_{\{u>0\}}\right)\,dx.
\end{equation}

Here is the almost monotonicity formula.

\begin{theorem}\label{thm : almost mon formula}
Let $u \geq 0$ be an almost minimizer to the functional \eqref{func:AC} in $B_2$. Assume further that
\[
	\omega_A(\cdot) + \omega_Q(\cdot) + \varrho(\cdot) \in C^{0,dini} \qquad \text{and} \qquad A(x_0) = I_d,
\]
for $x_0 \in F(u)\cap B_1$, and that
\[
	\int_{B_r(x_0)} |Du|^2\,dx \lesssim r^d, \quad \text{for} \quad r<1.
\]  
Then,
\[
    \mathcal{A}_{u,x_0}(r) \coloneqq r^{-d}\mathcal{J}_{x_0}(u,B_r(x_0)) - \frac{1}{2}r^{-d - 1}\int_{\partial B_r(x_0)}u^2\, d\mathcal{H}^{d-1},
\]
is almost monotone, in the sense that there exists a nonnegative real function $g \in L^1(0,1)$ such that
$$
	\frac{d}{dr} \mathcal{A}_{u,x_0}(r) \geq -g(r).
$$
\end{theorem}

\begin{proof}
Fix $x_0\in F(u)$ and assume, without loss of generality, that $x_0=0$ and $B_r\Subset B_1$ for every $r\in(0,1)$.
Set
\[
    A_0 \coloneqq A(0) = I_d, \qquad Q_0\coloneqq Q(0)>0,
\]
and define the \emph{frozen} functional at $0$ by
\[
    \mathcal J_{0}(v,E)\coloneqq \int_{E}\left(\frac12|\nabla v|^2+Q_0\,\chi_{\{v>0\}}\right)\,dx, \qquad E\subset \mathbb R^d.
\]
Then the Weiss-type quantity becomes
\[
    \mathcal A_{u,0}(r)=r^{-d}\mathcal J_0(u,B_r)-\frac12\,r^{-d-1}\int_{\partial B_r}u^2\,d\mathcal H^{d-1}.
\]
We calculate $\frac{d}{dr} A_{u,0}(r)$. Using the co-area formula, we have
\[
    \frac{d}{dr}\,\mathcal J_0(u,B_r)=\mathcal J_0(u,\partial B_r)
    \coloneqq \int_{\partial B_r}\left(\frac12|\nabla u|^2+Q_0\,\chi_{\{u>0\}}\right)\,d\mathcal H^{d-1},
\]
and so a direct computation yields
\begin{align*}
    \frac{d}{dr}\mathcal A_{u,0}(r)
    &= -d\,r^{-d-1}\mathcal J_0(u,B_r)+r^{-d}\mathcal J_0(u,\partial B_r) \\
    &\quad +\frac12(d+1)\,r^{-d-2}\int_{\partial B_r}u^2\,d\mathcal H^{d-1}
    - r^{-d-1}\int_{\partial B_r}u\,\partial_\nu u\,d\mathcal H^{d-1}.
\end{align*}
On $\partial B_r$, decompose $\nabla u = (\partial_\nu u)\,\nu + D_\tau u$, so that $|\nabla u|^2=|\partial_\nu u|^2+|D_\tau u|^2$.
By elementary algebra,
\[
    \frac12|\partial_\nu u|^2-\frac{u\,\partial_\nu u}{r}+\frac12\frac{u^2}{r^2}
    =\frac12\Big(\partial_\nu u-\frac{u}{r}\Big)^2.
\]
Collecting terms, we obtain
\begin{equation}\label{eq:Weiss-derivative}
    \frac{d}{dr}\mathcal A_{u,0}(r) = \mathcal B(u)-d\,r^{-d-1}\mathcal J_0(u,B_r)+\mathcal H(u),
\end{equation}
where
\begin{align*}
\mathcal B(u)
&\coloneqq r^{-d}\int_{\partial B_r}\Big(\tfrac12|D_\tau u|^2+Q_0\,\chi_{\{u>0\}}\Big)\,d\mathcal H^{d-1}
+\frac12\,r^{-d-1}\int_{\partial B_r}u^2\,d\mathcal H^{d-1},\\
\mathcal H(u)
&\coloneqq \frac12\,r^{-d-1}\int_{\partial B_r}\Big(\partial_\nu u-\frac{u}{r}\Big)^2\,d\mathcal H^{d-1}\ge 0.
\end{align*}
Let $w$ be the $1$-homogeneous extension of $u|_{\partial B_r}$ to $B_r$, namely
\[
    w(x)\coloneqq \frac{|x|}{r}\,u\!\left(r\frac{x}{|x|}\right)\quad \text{for} \quad x \not = 0, \quad \text{and}\quad w(0)\coloneqq 0.
\]
Then, since $w$ is $1$-homogeneous, $\mathcal A_{w,0}(s)$ is constant for $0<s\le r$; in particular,
\[
    0 = \frac{d}{ds}\mathcal A_{w,0}(s)\Big|_{s=r}.
\]
Applying \eqref{eq:Weiss-derivative} to $w$ and using that $\mathcal H(w) = 0$, we get
\[
    0 = \mathcal B(w)-d\,r^{-d-1}\mathcal J_0(w,B_r).
\]
Since $w=u$ on $\partial B_r$, we also have $\mathcal B(u)=\mathcal B(w)$, hence
\begin{equation}\label{eq:B-u-frozen-w}
\mathcal B(u)=d\,r^{-d-1}\mathcal J_0(w,B_r).
\end{equation}
Now we compare with the original functional. By the almost minimality of $u$, for every competitor $v$ with $v-u\in H^1_0(B_r)$,
\[
	\mathcal J(u,B_r)\le (1+\varrho(r))\,\mathcal J(v,B_r).
\]
Taking $v=w$, and using $\mathcal J\ge 0$, we obtain
\begin{equation}\label{eq:almost-min-lower}
	\mathcal J(w,B_r)\ge (1-\varrho(r))\,\mathcal J(u,B_r),
\end{equation}
since $(1+\varrho)^{-1}\ge 1-\varrho$ for $\varrho\ge 0$. Next, we compare 
\[
\begin{aligned}
	\mathcal J_0(v,B_r)-\mathcal J(v,B_r)
	&=\frac12\int_{B_r}\big\langle (I_d-A(x))\nabla v,\,\nabla v\big\rangle\,dx \\
	&\quad +\int_{B_r}\big(Q_0-Q(x)\big)\,\chi_{\{v>0\}}\,dx.
\end{aligned}
\]
Combining this with \eqref{eq:B-u-frozen-w} and \eqref{eq:almost-min-lower} gives
\begin{align*}
	\mathcal B(u)
	&= d\,r^{-d-1}\mathcal J_0(w,B_r) \\
	&= d\,r^{-d-1}\mathcal J(w,B_r)
	+d\,r^{-d-1}\big(\mathcal J_0(w,B_r)-\mathcal J(w,B_r)\big)\\
	&\ge d\,r^{-d-1}(1-\varrho(r))\,\mathcal J(u,B_r)
	+d\,r^{-d-1}\big(\mathcal J_0(w,B_r)-\mathcal J(w,B_r)\big).
\end{align*}
Insert this lower bound for $\mathcal B(u)$ into \eqref{eq:Weiss-derivative} to obtain
\begin{align}
\frac{d}{dr}\mathcal A_{u,0}(r)
&\ge \mathcal H(u)
-d\,r^{-d-1}\varrho(r)\,\mathcal J(u,B_r)
+d\,r^{-d-1}\big(\mathcal J(u,B_r)-\mathcal J_0(u,B_r)\big)\notag\\
&\quad +d\,r^{-d-1}\big(\mathcal J_0(w,B_r)-\mathcal J(w,B_r)\big). \label{eq:key-ineq}
\end{align}
Finally, we estimate the error terms. For $x\in B_r$, by continuity of $A$ and $Q$,
\[
\|A(x)-I_d\|\le \omega_A(r),
\qquad
|Q(x)-Q_0|\le \omega_Q(r).
\]
Hence, for any $v$,
\[
\big|\mathcal J(v,B_r)-\mathcal J_0(v,B_r)\big|
\le \frac12\,\omega_A(r)\int_{B_r}|\nabla v|^2\,dx + \omega_Q(r)\,|B_r|.
\]
Applying this to $v=u$ and $v=w$ and using \eqref{eq:key-ineq} yields
\begin{align*}
\frac{d}{dr}\mathcal A_{u,0}(r)
&\ge \mathcal H(u)
-d\,r^{-d-1}\varrho(r)\,\mathcal J(u,B_r)\\
& - C\,r^{-d-1}\omega_A(r)\!\int_{B_r}\big(|\nabla u|^2+|\nabla w|^2\big)\,dx
- C\,\frac{\omega_Q(r)}{r},
\end{align*}
for a dimensional constant $C$.

Finally, by our assumptions, we have
\[
\int_{B_r}|\nabla u|^2\,dx + \int_{B_r}Q(x)\chi_{\{u>0\}}\,dx \le C_0\,r^{d},
\qquad
\int_{B_r}|\nabla w|^2\,dx \le C_0\,r^{d}.
\]
Therefore,
\[
r^{-d-1}\varrho(r)\mathcal J(u,B_r)\le C\,\frac{\varrho(r)}{r},
\quad
r^{-d-1}\omega_A(r)\!\int_{B_r}\big(|\nabla u|^2+|\nabla w|^2\big)\,dx \le C\,\frac{\omega_A(r)}{r}.
\]
We conclude that
\[
\frac{d}{dr}\mathcal A_{u,0}(r)
\ge \mathcal H(u) - C\,\frac{\omega_A(r)+\omega_Q(r)+\varrho(r)}{r}
\ge -\,g(r),
\]
where
\[
g(r)\coloneqq C\,\frac{\omega_A(r)+\omega_Q(r)+\varrho(r)}{r}\ge 0.
\]
Since $\omega_A+\omega_Q+\varrho\in C^{0,\mathrm{dini}}([0,1])$, we have $\int_0^1 \frac{\omega_A(t)+\omega_Q(t)+\varrho(t)}{t}\,dt<\infty$, hence $g\in L^1(0,1)$. This proves the almost monotonicity of $\mathcal A_{u,0}(r)$.
\end{proof}

\begin{corollary}
Let $u \geq 0$ be a local minimizer to the functional \eqref{func:AC} in $B_2$. Assume further that
\[
	\omega_A(\cdot) + \omega_Q(\cdot) + \varrho(\cdot) \in C^{0,dini} \qquad \text{and} \qquad A(x_0) = I_d,
\]
for $x_0 \in F(u)\cap B_1$. Then, blow-ups are $1$-homogeneous.
\end{corollary}

\begin{remark}
We observe that we do not need the function $Q$ to be bounded from below. This is only necessary to ensure nondegeneracy. Also, for local minimizer, Lipschitz regularity is available by the regularity result from \cite{DSS1}.
\end{remark}

\subsection{The two-phase case}

Both Alt--Caffarelli and Alt--Phillips monotonicity formulas can be extended to the two-phase scenario, provided that optimal regularity is still available at the free boundary points where we do the blow-up procedure. For instance, in the Alt--Caffarelli scenarion, we may consider the following functional
\begin{equation}\label{two-phase functional AC}
	\mathcal J(u) \coloneqq \int \left( \frac{1}{2} \, \langle A(x)\nabla u, \nabla u \rangle + Q_1(x)\mathcal{X}_{\{u>0\}} + Q_2(x)\mathcal{X}_{\{u<0\}}\right)\,dx.
\end{equation}

\begin{theorem}\label{thm : almost mon formula 2}
Let $u$ be an almost minimizer to the functional \eqref{two-phase functional AC} in $B_2$. Assume further that
\[
	\omega_A(\cdot) + \omega_{Q_1}(\cdot) + \omega_{Q_2}(\cdot) + \varrho(\cdot) \in C^{0,dini} \qquad \text{and} \qquad A(x_0) = I_d,
\]
for $x_0 \in F(u)\cap B_1$ and that
\[
	\int_{B_r(x_0)} |Du|^2\,dx \lesssim r^d, \quad \text{for} \quad r<1.
\]
Then,
\[
    \mathcal{A}_{u,x_0}(r) \coloneqq r^{-d}\mathcal{J}_{x_0}(u,B_r(x_0)) - \frac{1}{2}r^{-d - 1}\int_{\partial B_r(x_0)}u^2\, d\mathcal{H}^{d-1},
\]
is almost monotone, in the sense that there exists a nonnegative real function $g \in L^1(0,1)$ such that
\[
	\frac{d}{dr} \mathcal{A}_{u,x_0}(r) \geq -g(r).
\]
\end{theorem}

\begin{proof}
The proof proceeds along the same lines as that of Theorem \ref{thm : almost mon formula}, with the only difference being that the frozen functional now takes the form
\[
\mathcal J_{0}(v,E)\coloneqq \int_{E}\left(\frac12|\nabla v|^2+Q_1(0)\chi_{\{v>0\}} +Q_2(0)\chi_{\{v<0\}} \right)\,dx, \qquad E\subset \mathbb R^d.
\]
The corresponding changes in the term $\mathcal{B}(u)$ are immediate, and the comparison between the original and frozen functionals yields one additional term.
\end{proof}

Similarly, in the Alt--Philips setting, one may consider the functional
\begin{equation}\label{two-phase functional AP}
    \mathcal{J}(w) \coloneqq \int \left( \frac{1}{2} \, \langle A(x)\nabla w, \nabla w \rangle + \delta_1(x)\,w_+^{\gamma(x)} + \delta_2(x)\,w_-^{\gamma(x)} \right)\,dx,
\end{equation}
and the corresponding almost-monotonicity formula is the following.

\begin{theorem}\label{thm:almost-monotonicity-3}
Let $u$ be an almost minimizer to the functional \eqref{def:main functional} in $B_2$. Assume further that $A(x_0) = I_d$ and
\[
	\omega_A(\cdot) + \omega_{\delta_1}(\cdot) + \omega_{\delta_2}(\cdot) + \varrho(\cdot) + |\ln(\cdot)|\omega_{\gamma}(\cdot) \in C^{0,dini} ([0,2]),
\]
and
\[
    \sup_{x\in B_r(x_0)}\Big(|u(x)|+|\nabla u(x)|^{\frac{2}{\gamma(x_0)}}\Big) \lesssim r^{\beta(x_0)},
\]
for some $x_0 \in F(u)\cap B_1$. Then, the function $\mathcal{A}_{u,x_0}(r)$ defined by
\[
     r^{-(d + 2(\beta(x_0) - 1))}\mathcal{J}_{x_0}(u,B_r(x_0)) - \frac{1}{2}\beta(x_0)r^{-(d + 2\beta(x_0) -1)}\int_{\partial B_r(x_0)}u^2\, d\mathcal{H}^{d-1},
\]
is almost monotone, meaning that there exists a nonnegative real function $g \in L^1(0,2)$ such that
\[
	\frac{d}{dr} \mathcal{A}_{u,x_0}(r) \geq -g(r) \quad \text{for} \quad r \in (0,1/2).
\]
\end{theorem}

\begin{proof}
Again, the proof is similar to the proof of Theorem \ref{thm : almost mon formula}, except for the frozen functional that becomes
\[
	\mathcal{J}_{x_0}(u,B_r(x_0)) \coloneqq \int_{B_r(x_0)}\left( \frac{1}{2} \, \langle |\nabla w|^2 + \delta_1(x_0)\,w_+^{\gamma(x_0)} + \delta_2(x_0)\,w_-^{\gamma(x_0)} \right)\,dx,
\]
the boundary term $\mathcal{B}(u)$ and the comparison of the frozen functional with the original one. 
\end{proof}

\begin{remark}
We remark that the growth estimate for $u$ at free boundary points may be hard to obtain, depending on whether $x_0$ is a two-phase branching free boundary point or not.
\end{remark}

\bigskip

{\small \noindent{\bf Acknowledgments.} This publication is based upon work supported by King Abdullah University of Science and Technology (KAUST).}

\medskip

\bibliographystyle{amsplain, amsalpha}

\end{document}